\documentclass[10pt, reqno]{amsart}
\usepackage{amsmath, amsthm, amscd, amsfonts, amssymb, graphicx, color}
\usepackage[bookmarksnumbered, colorlinks, plainpages]{hyperref}
\hypersetup{colorlinks=true,linkcolor=red, anchorcolor=green, citecolor=cyan, urlcolor=red, filecolor=magenta, pdftoolbar=true}
\usepackage{mathrsfs}

\newtheorem{theorem}{Theorem}[section]
\newtheorem{lemma}[theorem]{Lemma}

\newtheorem{corollary}[theorem]{Corollary}
\theoremstyle{definition}
\newtheorem{definition}[theorem]{Definition}
\newtheorem{example}[theorem]{Example}

\theoremstyle{remark}
\newtheorem{remark}[theorem]{Remark}
\numberwithin{equation}{section}

\begin{document}
\setcounter{page}{1}

\title[Quantum arithmetic dynamics]{Quantum arithmetic dynamics}

\author[Nikolaev]
{Igor V. Nikolaev$^1$}

\address{$^{1}$ Department of Mathematics and Computer Science, St.~John's University, 8000 Utopia Parkway,  
New York,  NY 11439, United States.}
\email{\textcolor[rgb]{0.00,0.00,0.84}{igor.v.nikolaev@gmail.com}}





\begin{abstract}
We study dynamics of  the Latt\`es maps  in the complex plane 
in terms of the Cuntz-Krieger algebras associated to the endomorphisms
of the non-commutative tori. 
In particular,
it is shown that iterations of  the Latt\`es maps can be reduced  to 
the dynamics of  the subshifts of finite type.
Using such a reduction, we calculate the zeta function of the Latt\`es maps.  
\end{abstract}

\maketitle

\section{Introduction}
The aim of our note is an interplay between the  arithmetic dynamics
[Silverman 2007] \cite{S} and  the quantum dynamics, i.e. 
the crossed product $C^*$-algebras [Williams 2007] \cite{W}. 
Namely, we construct a functor from the category of  Latt\`es
dynamical systems [Silverman 2007] \cite[Section 6.4]{S} to such of the 
Cuntz-Krieger algebras  [Cuntz \& Krieger 1980] \cite{CuKr1},
see theorem \ref{thm1.1}. 
Such a functor reduces  complex dynamics of the  Latt\`es
maps to the symbolic dynamics 
 on the alphabet consisting of only  two symbols [Lind \& Marcus 1995]
\cite{LM}, see corollary \ref{cor1.2}. 
Our construction is based on the fundamental correspondence between 
elliptic curves and non-commutative tori \cite[Section 1.3]{N}.

\bigskip
Let $\mathscr{E}$ be an  elliptic curve, i.e.   the subset of the complex projective plane of the form
$\{(x,y,z)\in \mathbf{C}P^2 ~|~ y^2z=x^3+ax^2z+bxz^2+cz^3\}$,
where $a, b$ and $c$  are some constant complex numbers.
Recall that a  rational map $\phi: \mathbf{C}P^1\to  \mathbf{C}P^1$ of degree $d\ge 2$
is called the  {\it Latt\`es map} if there exist an elliptic curve $\mathscr{E}$, a morphism 
$\psi: \mathscr{E}\to \mathscr{E}$, and a finite covering $\pi: \mathscr{E} \to \mathbf{C}P^1$,
such that $\pi\circ\phi=\psi\circ\pi$. 
\begin{example}
Consider a double covering $\pi$ of  the $\mathbf{C}P^1$ defined by the 
involution $(x,y,z)\mapsto (x,-y,z)$ of $\mathscr{E}$. 
Let $\psi: \mathscr{E}\to \mathscr{E}$ be the duplication map
$(x,y,z)\mapsto (2x, 2y, 2z)$ commuting with the involution $\pi$. 
Taking projection of $\mathscr{E}$ to the first coordinate and 
using the Bachet duplication formulas,  one gets the Latt\`es map:
\begin{equation}\label{eq1.1}
\phi(x)={x^4-2bx^2-8cx+b^2-4ac\over x^3+4ax^2+bx+c}. 
\end{equation}
\end{example}

\bigskip
The non-commutative torus $\mathscr{A}_{\theta}$ is a universal $C^*$-algebra
$\langle u,v~|~vu=e^{2\pi i\theta}uv\rangle$, where $u,v$ are unitary operators and 
$\theta\in\mathbf{R}$ is a constant. The $\mathscr{A}_{\theta}$ is said to have real 
multiplication if $\theta\in \mathbf{Q}(\sqrt{D})$ is a quadratic irrationality. 
The non-commutative tori and elliptic curves are related by a fundamental correspondence
saying that there exists a covariant functor $\mathscr{F}: \{\mathscr{E}\}\to \{\mathscr{A}_{\theta}\}$ 
which maps the morphisms  between elliptic curves to the endomorphisms of  non-commutative
tori  \cite[Section 1.3]{N}. 
\begin{example}\label{exm1.2}
Let $D>1$ be an integer. 
If $\mathscr{E}_{CM}$ is an elliptic curve with complex multiplication by 
$\sqrt{-D}$, then $\mathscr{F}(\mathscr{E}_{CM})=\mathscr{A}_{RM}$,
where $\mathscr{A}_{RM}$
is the non-commutative torus with real multiplication by $\sqrt{D}$
\cite[Theorem 6.1.2]{N}. 
\end{example}

\bigskip
Recall that the two-dimensional Cuntz-Krieger algebra $\mathcal{O}_A$ is a universal 
$C^*$-algebra $\langle s_1, s_2 ~|~ s_1^*s_1 = a_{11}s_1s_1^*+a_{12}s_2s_2^*,
~ s_2^*s_2 = a_{21}s_1s_1^*+a_{22}s_2s_2^*, 
~ s_1s_1^*+s_2s_2^*=1\rangle$, 
where $A=\left(\begin{smallmatrix}  a_{11} & a_{12}\cr a_{21} & a_{22} \end{smallmatrix}\right)$ 
is a matrix with the non-negative  integer entries,  while 
$s_1, s_2$ are partial isometries and   $s_1^*, s_2^*$ are their conjugates. 
The $\mathcal{O}_A$ is a quantum dynamical system,
i.e. it can be represented as a crossed product $C^*$-algebra  [Williams 2007] \cite{W}. 
Namely,
\begin{equation}\label{eq1.2}
\mathcal{O}_A\otimes\mathscr{K}\cong \mathbb{A}_{\theta}\rtimes_{\sigma}\mathbf{Z}, 
\end{equation}
where $\mathscr{K}$ is the $C^*$-algebra of all compact operators on 
a Hilbert space and the crossed product at the RHS of (\ref{eq1.2})  is taken by a
shift endomorphism $\sigma$ defined by a transposed matrix $A^t$ corresponding to 
the stationary AF-algebra $\mathbb{A}_{\theta}$ containing a dense copy of
 $\mathscr{A}_{\theta}$, see e.g. [Blackadar 1986] \cite[Exercise 10.11.9 (b)]{B},
 \cite[Sections 3.5 and 3.7]{N} or Section 2.3 of this paper.

 \bigskip
 Let  $K\subset\mathbf{C}$ be a number field,  such that $K\cong \mathbf{Q}$ or a finite extension of $\mathbf{Q}$.   
 Let $PGL_2(K)=GL_2(K)/K^*$ be the matrix group with entries in $K$, see Section 2.1. 
 Denote by  $\mathscr{L}$ a category of all  dynamical systems
 generated by iterations of the  Latt\`es maps $\phi: \mathbf{C}P^1\to \mathbf{C}P^1$  with the coefficients in  $K$
 and such that  the projection maps $\pi$  
and $\pi'$  associated to $\phi$  and $\phi'$ in Figure 1  both have degree $2$. 
  The arrows (morphisms) 
 of  $\mathscr{L}$ are the linear conjugacies $\{f^{-1}\circ\phi\circ f~|~f\in PGL_2(K)\}$  between the Latt\`es 
 maps $\phi$. Likewise, denote by $\mathscr{O}$ a category of all two-dimensional Cuntz-Krieger 
 algebras $\mathcal{O}_A$. The arrows (morphisms)  of  $\mathscr{O}$ are 
 the Morita equivalencies of the  Cuntz-Krieger 
   algebras  $\mathcal{O}_{A'}\otimes\mathscr{K}\cong \mathcal{O}_{A}\otimes\mathscr{K}$,
   i.e. the similarity classes $\{A'=T^{-1}A T ~|~T\in GL_2(\mathbf{Z})\}$ of the matrices $A$.
   We define the  map $F: \mathscr{L}\to \mathscr{O}$ as a composition of the above
   maps:
\begin{equation}
(\mathbf{C}P^1,\phi)\buildrel\rm\pi^{-1}\over\longrightarrow 
(\mathscr{E}_{CM}, \psi) 
   \buildrel\rm \mathscr{F}\over\longrightarrow
 (\mathbb{A}_{RM}, \sigma)   
   \buildrel\rm (\ref{eq1.2})\over\longrightarrow 
   \mathcal{O}_A,
\end{equation}

\bigskip
  Our main results can be formulated as follows. 
\begin{theorem}\label{thm1.1}
The map $F: \mathscr{L}\to\mathscr{O}$ is a functor transforming the  conjugate
Latt\`es dynamical  systems $(\mathbf{C}P^1, \phi)$  to the Morita equivalent  Cuntz-Krieger algebras
$\mathcal{O}_A$. 
\end{theorem}
Denote by $(X_A,\sigma_A)$ a subshift of finite type corresponding to the
matrix $A$ [Lind \& Marcus 1995] \cite{LM}. 
Recall that the subshifts  $(X_A,\sigma_A)$ and 
 $(X_{A'},\sigma_{A'})$ are shift equivalent (over $\mathbb{Z}^+$) if there exist non-negative matrices $R$ and $S$
and a positive integer $k$ (a lag), satisfying the equations 
$AR=RA', A'S=SA, A^k=RS$ and $SR=(A')^k$. 
Theorem \ref{thm1.1}
implies a reduction of the complex dynamics of the Latt\`es maps 
to the symbolic dynamics of the subshifts of finite type.
(To the best of our knowledge, such a reduction is known only
for the H\'enon and real rational maps so far.)  
\begin{corollary}\label{cor1.2}
The Latt\`es dynamical systems 
 $(\mathbf{C}P^1, \phi)$  and $(\mathbf{C}P^1, \phi')$
are $PGL_2(K)$-conjugate if and only if $(X_A,\sigma_A)$ and $(X_{A'},\sigma_{A'})$ 
are shift equivalent. 
\end{corollary}
Let $\phi:\mathbf{C}P^1\to \mathbf{C}P^1$ be a Latt\`es map defined over 
the number field $K$. Denote by $Per_n(\phi)=\{x\in   \mathbf{C}P^1 ~|~\phi^n(x)=x\}$
the set of $n$-periodic points of the map $\phi$. Recall that the zeta function
of $\phi$ is defined by the formula:
\begin{equation}\label{eq1.3}
\zeta_{\phi}(t):=\exp \left(\sum_{n=1}^{\infty} |Per_n(\phi)|{t^n\over n}\right). 
\end{equation} 
\begin{corollary}\label{cor1.3}
Let $F(\phi)=\mathcal{O}_A$.  Denote by $tr (A)$  and $det (A)$ the trace and the determinant of matrix $A$,
respectively.  Then
\begin{equation}\label{eq1.4}
\zeta_{\phi}(t)={1\over det (A) ~t^2-tr (A) ~t+1}. 
\end{equation}
\end{corollary}
\begin{remark}
A general formula of the  zeta function for a rational map of  the $\mathbf{C}P^1$ was 
obtained by [Hinkkanen 1994] \cite{Hik1}.  Formula (\ref{eq1.4}) can be viewed a refinement of Hinkkanen's formula 
to the Latt\`es maps.  
\end{remark}
The article is organized as follows. 
The preliminary facts can be found in Section 2.
Theorem \ref{thm1.1} and corollaries \ref{cor1.2}-\ref{cor1.3}   are proved in Section 3. 
An example is considered in Section 4.

\section{Preliminaries}
We briefly review the Latt\`es maps, symbolic dynamics,
non-commutative tori and Cuntz-Krieger algebras. 
The reader is referred to   [Silverman 2007] \cite[Section 6.4]{S}, 
  [Lind \& Marcus 1995] \cite{LM},
 [Cuntz \& Krieger 1980] \cite{CuKr1}
and \cite[Section 1]{N} for a detailed exposition.

\subsection{Latt\`es maps}
Denote by  $\mathscr{E}$ an  elliptic curve, i.e.   the subset of the complex projective plane of the form
$\{(x,y,z)\in \mathbf{C}P^2 ~|~ y^2z=x^3+ax^2z+bxz^2+cz^3\}$,
where $a, b$ and $c$  are some constant complex numbers.
\begin{definition}
A rational map $\phi: \mathbf{C}P^1\to  \mathbf{C}P^1$ of degree $d\ge 2$ is called 
a Latt\`es map if there are an elliptic curve $\mathscr{E}$, a morphism 
$\psi: \mathscr{E}\to \mathscr{E}$, and a finite covering $\pi: \mathscr{E} \to \mathbf{C}P^1$
such that the diagram in Figure 1 is commutative. 
\end{definition}
\begin{figure}
\begin{picture}(300,110)(-80,-5)
\put(20,70){\vector(0,-1){35}}
\put(123,70){\vector(0,-1){35}}
\put(45,23){\vector(1,0){60}}
\put(45,83){\vector(1,0){60}}
\put(15,20){$\mathbf{C}P^1$}
\put(115,20){$\mathbf{C}P^1$}
\put(17,80){$ \mathscr{E}$}
\put(115,80){ $ \mathscr{E}$}
\put(70,30){$\phi$}
\put(70,90){$\psi$}
\put(7,50){$\pi$}
\put(132,50){$\pi$}
\end{picture}
\caption{Latt\`es map}
\end{figure}
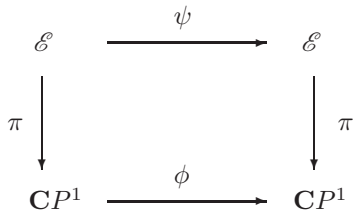

Recall that the M\"obius transformation is a map 
 $\mathbf{C}P^1\to  \mathbf{C}P^1$
of the form: 
\begin{equation}\label{eq2.1}
z\mapsto{az+b\over cz+d}, \quad ad-bc\ne 0, \quad a,b,c,d\in \mathbf{C}.  
\end{equation}
Such a transformation defines an automorphism of  $\mathbf{C}P^1$
and each automorphism of  $\mathbf{C}P^1$ has the form (\ref{eq2.1}). 
The M\"obius transformations can be represented by matrices 
$\left(\begin{smallmatrix} a & b\cr c & d\end{smallmatrix}\right)$ 
and they make the matrix group $GL_2(\mathbf{C})$ under the composition 
operation. Two matrices define the same M\"obius transformation 
if and only if they are scalar multiples of one another. Thus
the automorphism group of  $\mathbf{C}P^1$ is given by the formula:
\begin{equation}\label{eq2.2}
Aut ~\mathbf{C}P^1\cong GL_2(\mathbf{C})/\mathbf{C}^* \cong PGL_2(\mathbf{C}). 
\end{equation}
\begin{definition}
Two rational maps $\phi$ and $\phi'$ of $\mathbf{C}P^1$ are said to be 
$PGL_2(\mathbf{C})$-conjugate if for some $f\in  PGL_2(\mathbf{C})$ is holds 
\begin{equation}\label{eq2.3}
\phi'=f^{-1}\circ\phi\circ f. 
\end{equation}
\end{definition}
\begin{remark}
In view of (\ref{eq2.2}),  the $PGL_2(\mathbf{C})$-conjugate rational maps 
correspond to  the same map of   $\mathbf{C}P^1$ up to a coordinate change. 
Clearly, the iterations of such maps generate the same dynamical systems.    
\end{remark}

Let $\mathbf{Q}\subseteq K\subset\mathbf{C}$ be an algebraic number field, i.e. a finite degree extension of the  field $\mathbf{Q}$. 
The following result is proved in  [Silverman 2007] \cite[Theorem 6.46]{S}. 
\begin{theorem}\label{thm2.4}
Let $\phi$  and $\phi'$  be Latt\`es maps defined over $K$  that are associated, respectively,
to elliptic curves $\mathscr{E}$ and $\mathscr{E}'$.  Assume  that the projection maps $\pi$  
and $\pi'$  associated to $\phi$  and $\phi'$  both have degree $2$. 
If $\phi$  and $\phi'$ are  $PGL_2 (K)$-conjugate to
one another, then $\mathscr{E}$ and $\mathscr{E}'$ are isomorphic.
\end{theorem}

\subsection{Subshifts of finite type}
A full Bernoulli $n$-shift is the set $X_n$ of bi-infinite sequences
$x=\{x_k\}$, where $x_k$ is a symbol taken from a set $S$ of cardinality $n$.  
The set $X_n$ is endowed with the product topology, making $X_n$ a Cantor set.
The shift homeomorphism $\sigma_n:X_n\to X_n$ is given by the formula
$\sigma_n(\dots x_{k-1}x_k x_{k+1}\dots)=(\dots x_{k}x_{k+1}x_{k+2}\dots)$
The homeomorphism defines a dynamical system $\{X_n,\sigma_n\}$
given by the iterations of $\sigma_n$. 

Let $A$ be an $n\times n$ matrix, whose entries $a_{ij}:= a(i,j)$ are $0$ or $1$.
Consider a subset $X_A\subseteq X_n$ consisting of the bi-infinite sequences,
which satisfy the restriction $a(x_k, x_{k+1})=1$ for all $-\infty<k<\infty$.
(It takes a moment to verify that $X_A$ is indeed a subset of $X_n$ and 
$X_A=X_n$, if and only if,  all the entries of $A$ are $1$'s.) By definition,
$\sigma_A=\sigma_n~|~X_A$ and the pair $\{X_A,\sigma_A\}$ is called a 
subshift of finite type (SFT).
A standard edge shift construction described in [Lind \& Marcus 1995]  \cite{LM}
allows to extend the notion of SFT to any matrix $A$ with the non-negative entries.

It is well known that the SFT's $\{X_A,\sigma_A\}$ and $\{X_B,\sigma_B\}$
are topologically conjugate (as the dynamical systems), if and only if, 
the matrices $A$ and $B$ are  strong shift equivalent (SSE), see 
[Lind \& Marcus 1995]  \cite{LM} for 
the corresponding definition. The SSE of two matrices is a difficult
algorithmic problem, which motivates the consideration of a weaker 
equivalence between the matrices called a  shift equivalence (SE). 
Namely,  the matrices $A$ and $B$ are said to be shift equivalent
(over $\mathbf{Z}^+$),  when there exist non-negative matrices $R$ and $S$
and a positive integer $k$ (a lag), satisfying the equations 
$AR=RB, BS=SA, A^k=RS$ and $SR=B^k$. Finally, the SFT's    
 $\{X_A,\sigma_A\}$ and $\{X_B,\sigma_B\}$ (and the matrices $A$ and $B$)
are said to be flow equivalent (FE),  if the suspension flows of the SFT's
act on the topological spaces, which are homeomorphic under a homeomorphism
that respects the orientation of the orbits of the suspension flow. 
We shall use the following implications
\begin{equation}
SSE\Rightarrow  SE\Rightarrow  FE.
\end{equation}

\subsection{$C^*$-agebras}
The $C^*$-algebra is an algebra  $\mathscr{A}$ over $\mathbf{C}$ with a norm 
$a\mapsto ||a||$ and an involution $\{a\mapsto a^* ~|~ a\in \mathscr{A}\}$  such that $\mathscr{A}$ is
complete with  respect to the norm, and such that $||ab||\le ||a||~||b||$ and $||a^*a||=||a||^2$ for every  $a,b\in \mathscr{A}$.  
Each commutative $C^*$-algebra is  isomorphic
to the algebra $C_0(X)$ of continuous complex-valued
functions on some locally compact Hausdorff space $X$. 
Any other  algebra $\mathscr{A}$ can be thought of as  a noncommutative  
topological space.

\subsubsection{AF-algebras}
An  AF-algebra  (approximately finite $C^*$-algebra) is defined to
be the  norm closure of an ascending sequence of the finite dimensional
$C^*$-algebras $M_n$'s, where  $M_n$ is the $C^*$-algebra of the $n\times n$ matrices
with the entries in $\mathbf{C}$. Here the index $n=(n_1,\dots,n_k)$ represents
a semi-simple matrix algebra $M_n=M_{n_1}\oplus\dots\oplus M_{n_k}$.
The ascending sequence can be written as 
$M_1\buildrel\rm\varphi_1\over\longrightarrow M_2
   \buildrel\rm\varphi_2\over\longrightarrow\dots,
$
where $M_i$ are the finite dimensional $C^*$-algebras and
$\varphi_i$ the homomorphisms between such algebras. 
(It is easy to see, that
each $\varphi_i$ is given by an integer matrix $A_i$ with non-negative entries.)
 The set-theoretic limit
$\mathbb{A}=\lim M_n$ has a natural algebraic structure given by the formula
$a_m+b_k\to a+b$; here $a_m\to a,b_k\to b$ for the
sequences $a_m\in M_m,b_k\in M_k$.

If the homomorphisms $\varphi_1 =\varphi_2=\dots=Const$ then
 the AF-algebra $\mathbb{A}$ is called  stationary. 
In this case, the AF-algebra is given by a single matrix $A$ and by 
taking a power of $A$ one gets a strictly positive integer matrix,   which 
we always assume to be the case.  The stationary AF-algebra  $\mathbb{A}$
is therefore the limit 
$M_1\buildrel\rm A\over\longrightarrow M_2
   \buildrel\rm A\over\longrightarrow\dots,
$
and shifting the sequence by one generates a map $\sigma: \mathbb{A}\to \mathbb{A}$
called a shift endomorphism of $\mathbb{A}$.

\subsubsection{Dimension groups}
By $M_{\infty}(\mathscr{A})$ 
one understands the algebraic direct limit of the $C^*$-algebras 
$M_n(\mathscr{A})$ under the embeddings $a\mapsto ~\mathbf{diag} (a,0)$. 
The direct limit $M_{\infty}(\mathscr{A})$  can be thought of as the $C^*$-algebra 
of infinite-dimensional matrices whose entries are all zero except for a finite number of the
non-zero entries taken from the $C^*$-algebra $\mathscr{A}$.
Two projections $p,q\in M_{\infty}(\mathscr{A})$ are equivalent, if there exists 
an element $v\in M_{\infty}(\mathscr{A})$,  such that $p=v^*v$ and $q=vv^*$. 
The equivalence class of projection $p$ is denoted by $[p]$.   
We write $V(\mathscr{A})$ to denote all equivalence classes of 
projections in the $C^*$-algebra $M_{\infty}(\mathscr{A})$, i.e.
$V(\mathscr{A}):=\{[p] ~:~ p=p^*=p^2\in M_{\infty}(\mathscr{A})\}$. 
The set $V(\mathscr{A})$ has the natural structure of an abelian 
semi-group with the addition operation defined by the formula 
$[p]+[q]:=\mathbf{diag}(p,q)=[p'\oplus q']$, where $p'\sim p, ~q'\sim q$ 
and $p'\perp q'$.  The identity of the semi-group $V(\mathscr{A})$ 
is given by $[0]$, where $0$ is the zero projection. 
By the $K_0$-group $K_0(\mathscr{A})$ of the unital $C^*$-algebra $\mathscr{A}$
one understands the Grothendieck group of the abelian semi-group
$V(\mathscr{A})$, i.e. a completion of $V(\mathscr{A})$ by the formal elements
$[p]-[q]$.  The image of $V(\mathscr{A})$ in  $K_0(\mathscr{A})$ 
is a positive cone $K_0^+(\mathscr{A})$ defining  the order structure $\le$  on the  
abelian group  $K_0(\mathscr{A})$.

The pair   $\left(K_0(\mathscr{A}),  K_0^+(\mathscr{A})\right)$
is known as a dimension group of the $C^*$-algebra $\mathscr{A}$. 
The scale $\Sigma(\mathscr{A})$ is the image in $K_0^+(\mathscr{A})$
of the equivalence classes of projections in the $C^*$-algebra $\mathscr{A}$. 
 Each  scale  can always be written as 
$\Sigma(\mathscr{A})=\{a\in K_0^+(\mathscr{A}) ~|~0\le a\le u\}$,
where $u$ is an  order unit of  $K_0^+(\mathscr{A})$.  
The pair  $\left(K_0(\mathscr{A}),  K_0^+(\mathscr{A})\right)$ and the
triple  $\left(K_0(\mathscr{A}),  K_0^+(\mathscr{A}), \Sigma(\mathscr{A})\right)$
are invariants of the Morita equivalence and isomorphism class of the 
$C^*$-algebra $\mathscr{A}$, respectively.

If  $\mathbb{A}$ is an AF-algebra, then its scaled dimension group 
(dimension group, resp.) is a complete invariant of the isomorphism 
(Morita equivalence, resp.) class of $\mathbb{A}$, see e.g. \cite[Theorem 3.5.2]{N}.

\subsubsection{Non-commutative tori}
By a non-commutative torus one understands the universal $C^*$-algebra 
$\mathscr{A}_{\theta}$ generated by unitaries $u$ and $v$ satisfying 
the commutation relation $vu=e^{2\pi i\theta}uv$ for a real constant $\theta$. 
The $\mathscr{A}_{\theta}$ is said to have real 
multiplication if $\theta\in \mathbf{Q}(\sqrt{D})$ is a quadratic irrationality.

The non-commutative tori and elliptic curves are related by a fundamental correspondence
saying that there exists a covariant functor $\mathscr{F}: \{\mathscr{E}\}\to \{\mathscr{A}_{\theta}\}$ 
which maps the morphisms  between elliptic curves to the endomorphisms of  non-commutative
tori   \cite[Section 1.3]{N}.

\subsubsection{Effros-Shen algebras}
The Effros-Shen algebra $\mathbb{A}_{\theta}$ is an AF-algebra 
given by the matrices 
\begin{equation}
A_i=\prod_{j=0}^i\left(\begin{matrix} a_j & 1\cr 1 & 0
\end{matrix}\right),
\end{equation}
where $\theta=[a_0, a_1, \dots]$
 is the  continued fraction of a real number $\theta$. 
The Pimsner-Voiculescu Theorem says that  $\mathscr{A}_{\theta}$
is a dense subalgebra of the Effros-Shen algebra $\mathbb{A}_{\theta}$
under an embedding
\begin{equation}\label{eq2.5}
\mathscr{A}_{\theta}\hookrightarrow \mathbb{A}_{\theta}. 
\end{equation}

\subsubsection{Cuntz-Krieger algebras}
A  Cuntz-Krieger algebra, $\mathcal{O}_A$, is the $C^*$-algebra
generated by partial isometries $s_1,\dots, s_n$ that act on a Hilbert space in such 
a way that their support projections $Q_i=s_i^*s_i$ and their  range projections
$P_i=s_is_i^*$ are orthogonal and satisfy the relations
$Q_i=\sum_{j=i}^nb_{ij}P_j$,
for an $n\times n$  matrix $A=(a_{ij})$ consisting of $0$'s 
and $1$'s [Cuntz  \&  Krieger 1980]  \cite{CuKr1}.  The notion is extendable to the 
matrices $A$ with the non-negative integer entries 
[Cuntz  \&  Krieger 1980]  \cite[Remark 2.16]{CuKr1}.
It is known,  that the $C^*$-algebra $\mathcal{O}_A$ is simple, 
whenever matrix $A$ is irreducible (i.e. a certain power
of $A$ is a strictly positive integer matrix). 
It was established in [Cuntz  \&  Krieger 1980]  \cite{CuKr1},  that $K_0(\mathcal{O}_A)\cong \mathbf{Z}^n / (I-A^t)\mathbf{Z}^n$
and $K_1(\mathcal{O}_A)=Ker~(I-A^t)$, where $A^t$ is a transpose of the matrix $A$.
It is not difficult to see, that whenever $det~(I-A^t)\ne 0$, the $K_0(\mathcal{O}_A)$
is a finite abelian group and $K_1(\mathcal{O}_A)=0$.  The both groups are invariants of the
stable isomorphism class of the Cuntz-Krieger algebra.

Let  $\mathbb{A}_{\theta}$ be  the stationary Effros-Shen algebra given by a matrix $A$
and let $\sigma$ be the a shift endomorphism of $\mathbb{A}_{\theta}$. Denote by  $\mathscr{K}$  
the $C^*$-algebra of compact operators. 
The Cuntz-Krieger algebra  can be written as the crossed product $C^*$-algebra:
\begin{equation}\label{eq2.6}
\mathcal{O}_A\otimes\mathscr{K}\cong \mathbb{A}_{\theta}\rtimes_{\sigma}\mathbf{Z}.
\end{equation} 

\section{Proof}
\subsection{Proof of theorem \ref{thm1.1} }
We shall split the proof in two lemmas.
\begin{lemma}\label{lm3.1}
Let $\mathbb{A}_{RM}$ be the Effros-Shen algebra containing a
non-commutative torus $\mathscr{A}_{RM}$ given by  the embedding (\ref{eq2.5}). 
For an endomorphism $\epsilon\in End~\mathbb{A}_{RM}$ there
exists a stationary AF-algebra $\mathbb{A}\subseteq\mathbb{A}_{RM}$ given by 
a  matrix $T\in GL_2(\mathbf{Z})$,
such that: 
\begin{equation}\label{eq3.1}
\mathbb{A}_{RM}\rtimes_{\epsilon}\mathbf{Z}\cong 
\mathbb{A}\rtimes_{\sigma}\mathbf{Z}. 
\end{equation}
where $\sigma$ is  the shift automorphism of $\mathbb{A}$.
 \end{lemma}
\begin{proof}
Roughly speaking, the isomorphism (\ref{eq3.1}) follows from the 
Unimodular Conjecture for the dimension groups [Effros 1981] \cite[p.34]{E}. Such a conjecture 
is known to be true for the dimension groups associated to the  Effros-Shen algebras.  
We refer the reader to [Effros 1981] \cite{E} for the notations  and details.
Let us proceed step by step.   

\bigskip
(i) Let $\Lambda=\mathbf{Z}+\mathbf{Z}\theta\subset\mathbf{R}$ and $\Lambda^+=\{\Lambda ~|~ \mathbf{Z}+\mathbf{Z}\theta>0\}$.
Denote by $(\Lambda, \Lambda^+)$ a dimension group associated to the Effros-Shen algebra  $\mathbb{A}_{RM}$. 
Since our Effros-Shen algebra has real multiplication, we conclude that $\Lambda\subset k$, where $k=\mathbf{Q}(\sqrt{D})$ is a real quadratic 
number field. 

\bigskip
(ii) Let $O_k$ be the ring of integers of the field $k$. 
The multiplication of  $\Lambda$ by an element of $O_k$
generates  an endomorphism of $\Lambda$. 
It is not hard to see,  that 
 $End~\Lambda\cong O_k$ or to an order in the ring $O_k$.

\bigskip
(iii) Since each $\epsilon:  \mathbb{A}_{RM}\to \mathbb{A}_{RM}$ induces 
 an endomorphism of $\Lambda$, we have 
$End ~\mathbb{A}_{RM}\cong End ~\Lambda$.

\bigskip
(iv) Let us  determine matrix $T\in GL_2(\mathbf{Z})$ in lemma \ref{lm3.1}. 
Denote by  $A\in M_2(\mathbf{Z})$  the matrix form of the algebraic integer  $\epsilon\in O_k$
under the isomorphism 
\linebreak
 $End ~\mathbb{A}_{RM}\cong O_k$. 
Consider a dimension group given by the inductive limit:
\begin{equation}\label{eq3.2}
\mathbf{Z}^2\buildrel\rm A\over\longrightarrow \mathbf{Z}^2
   \buildrel\rm A\over\longrightarrow\dots
\end{equation}

\bigskip
(v) The dimension group (\ref{eq3.2}) 
defines a stationary  AF-algebra $\mathbb{A}\subseteq\mathbb{A}_{RM}$
corresponding to the dimension group $(\Lambda_A, \Lambda_A^+)
\subseteq (\Lambda, \Lambda^+)$ invariant under iterations of the endomorphism $A$.  
The index of $\Lambda_A$ in $\Lambda$ is equal to the degree $|det~A|$ of the endomorphism
$\epsilon$.

\bigskip
(vi) Since the Unimodular Conjecture is true for the Effros-Shen algebras   
[Effros 1981] \cite[p.34]{E}, we conclude that inductive limit (\ref{eq3.2}) 
can be written as: 
\begin{equation}\label{eq3.3}
\mathbf{Z}^2\buildrel\rm T\over\longrightarrow \mathbf{Z}^2
   \buildrel\rm T\over\longrightarrow\dots, 
\end{equation}
where $T\in GL_2(\mathbf{Z})$.  

\bigskip
(vii) Our matrix $T$ can be obtained from the continued fraction of $\epsilon\in O_k$
as follows.  
Since $\epsilon$ is a quadratic irrationality, the corresponding continued fraction must
be eventually $k$-periodic for some $k\ge 1$, i.e.  
$\epsilon=[b_1,\dots, b_N;  ~\overline{a_1,\dots,a_k}]$.
Consider the following inductive limit:
\begin{equation}\label{eq3.4}
\mathbf{Z}^2\buildrel\rm 
\left(
\begin{smallmatrix}
b_1 & 1\cr 1 & 0
\end{smallmatrix}
\right) 
\over\longrightarrow 
\dots
\buildrel\rm 
\left(
\begin{smallmatrix}
b_N & 1\cr 1 & 0
\end{smallmatrix}
\right) 
\over\longrightarrow 
\mathbf{Z}^2
   \buildrel\rm A\over\longrightarrow\dots, 
\end{equation}
where
\begin{equation}\label{eq3.5}
T=
\left(
\begin{matrix}
a_1 & 1\cr 1 & 0
\end{matrix}
\right)
\dots
 \left(
\begin{matrix}
a_k & 1\cr 1 & 0
\end{matrix}
\right). 
\end{equation}

\bigskip
(viii) 
Since  (\ref{eq3.3}) and (\ref{eq3.4}) differ only in a finite number of terms, 
we conclude that the corresponding inductive limits define the same dimension 
group $(\Lambda_A, \Lambda_A^+) \subseteq (\Lambda, \Lambda^+)$. 
Thus  $T$ in (\ref{eq3.3}) is given by the formula (\ref{eq3.5}). 

\bigskip
(ix)  The isomorphism (\ref{eq3.1})  follows from  such 
between the dimension groups (\ref{eq3.2}) and (\ref{eq3.3}).

\bigskip
This argument finishes the proof of lemma \ref{lm3.1}. 
\end{proof}

\bigskip
\begin{corollary}\label{cor3.2}
$\mathbb{A}_{RM}\rtimes_{\epsilon}\mathbf{Z}\cong  \mathcal{O}_A\otimes\mathscr{K}$.
\end{corollary}
\begin{proof}
The corollary follows from formulas (\ref{eq2.6}) and (\ref{eq3.1}). 
\end{proof}

\bigskip
\begin{lemma}\label{lm3.3}
If $\phi$ and $\phi'$ are the $PGL_2(K)$-conjugate 
Latt\`es maps, then the corresponding Cuntz-Krieger 
algebras $\mathcal{O}_A$ and $\mathcal{O}_{A'}$ 
are Morita equivalent. 
\end{lemma}
\begin{proof}
In outline, the proof of lemma \ref{lm3.3} follows from Theorem \ref{thm2.4}
and corollary \ref{cor3.2}.  Let us pass to a detailed argument.

\bigskip
(i)  Let $\phi$ and $\phi'$ be  the $PGL_2(K)$-conjugate 
Latt\`es maps lying in the category $\mathscr{L}$. 
Since the covering maps $\pi$ and $\pi'$  both have degree $2$,
one can apply Theorem \ref{thm2.4}. 

\bigskip
(ii) It follows from \ref{thm2.4} that there exists an isomorphism $\psi:\mathscr{E}\to
\mathscr{E}'$ between  the corresponding elliptic curves.

\bigskip
(iii) The isomorphism $\psi$ induces an isomorphism of the non-commutative 
tori $\mathscr{A}_{RM}\to \mathscr{A}_{RM}'$.  The same is true for the 
Effros-Shen algebras    $\mathbb{A}_{RM}\to \mathbb{A}_{RM}'$
under the embedding (\ref{eq2.5}). 

\bigskip
(iv) We conclude from (iii) that the crossed product $C^*$-algebras
$\mathbb{A}_{RM}\rtimes_{\epsilon}\mathbf{Z}$ and 
$\mathbb{A}_{RM}'\rtimes_{\epsilon}\mathbf{Z}$. must be isomorphic. 

\bigskip
(v)  From (iv) and corollary \ref{cor3.2} one gets an isomorphism
of the $C^*$-algebras   $\mathcal{O}_A\otimes\mathscr{K}\cong
 \mathcal{O}_{A'}\otimes\mathscr{K}$. In other words, the Cuntz-Krieger 
 algebras $\mathcal{O}_A$ and $\mathcal{O}_{A'}$ are Morita
 equivalent. 
 
 \bigskip
This argument finishes the proof of lemma \ref{lm3.3}. 
 \end{proof}

\bigskip
Theorem \ref{thm1.1} follows from lemma \ref{lm3.3}.

\subsection{Proof of corollary \ref{cor1.2}}
We split the proof in the following steps.

\bigskip
(i) Let $\phi, \phi'\in\mathscr{L}$ be conjugate Latt\`es map. Denote by $\mathcal{O}_A, \mathcal{O}_{A'}\in\mathscr{O}$ 
the corresponding Cuntz-Krieger algebras (theorem \ref{thm1.1}). 

\bigskip
(ii) Using lemma \ref{lm3.1} when $\epsilon$ is an automorphism of $\mathbb{A}_{RM}$,
 we conclude that  dimension groups $(\Lambda_A, \Lambda^+_A)$ and $(\Lambda_{A'}, \Lambda^+_{A'})$
are  isomorphic.

\bigskip
(iii)  In view of (ii),  one can apply Krieger's Theorem ([Wagoner 1999] \cite[Theorem 2.11]{Wag1})
saying that  $(X_A, \sigma_A)$ and $(X_{A'}, \sigma_{A'})$ are shift 
equivalent over $\mathbf{Z}^+$ if and only if 
 $(\Lambda_A, \Lambda^+_A)$ and $(\Lambda_{A'}, \Lambda^+_{A'})$
are  isomorphic dimension groups.

\bigskip
(iv) The converse statement is proved similarly  and is left to the reader.

 \bigskip
 Corollary \ref{cor1.2} follows from items (i)-(iv).

\subsection{Proof of corollary \ref{cor1.3}}
Roughly speaking, corollary \ref{cor1.3} follows from the Lefschetz fixed-point theorem 
applied to  the compact topological space $X_A$ and  corollary \ref{cor1.2}. 
Let us pass to a step by step argument. 

\bigskip
(i) In view of corollary \ref{cor1.2}, the dynamics of the Latt\`es 
map $\phi$ is encoded by the subshift $(X_A, \sigma_A)$, where 
$\mathcal{O}_A=F(\phi)$.  
In particular, the set $Per_n(\phi)$ has the same cardinality as the set $Per_n(\sigma_A)$
of all $n$-periodic points $x\in X_A$.

\bigskip
(ii) It is well known, that $|Per_n(\sigma_A)|=tr(A^n)$, see e.g. [Lind \& Marcus 1995]\cite[p.195]{LM}
or [Wagoner 1999] \cite[p. 273]{Wag1}.  This formula follows from the Lefschetz fixed-point theorem 
applied to  the compact topological space $X_A$.

\bigskip
(iii) Let $\lambda$ and $\bar\lambda$ be the complex conjugate eigenvalues of the matrix $A$. 
 Bringing $A$ to the diagonal form 
$A=\left(\begin{smallmatrix}  \lambda & 0\cr 0 & \bar\lambda\end{smallmatrix}\right)$,
one gets $A^n=\left(\begin{smallmatrix}  \lambda^n & 0\cr 0 & \bar\lambda^n\end{smallmatrix}\right)$
and therefore $tr(A^n)=\lambda^n+\bar\lambda^n$.

\bigskip
(iv)  One can write (\ref{eq1.3}) in the form:
\begin{equation}\label{eq3.6}
\zeta_{\phi}(t)=\exp \left(\sum_{n=1}^{\infty} {(\lambda t)^n+(\bar\lambda t)^n\over n}\right)=
\exp \left(\ln {1\over (1-\lambda t)(1-\bar\lambda t)}\right),  
\end{equation} 
where the second equality follows from the formula for the sum of geometric progression  and 
formal integration of the obtained series.  On gets  from (\ref{eq3.6}):   
\begin{equation}\label{eq3.7}
\zeta_{\phi}(t)=
 {1\over (1-\lambda t)(1-\bar\lambda t)}=
 {1\over det ~(I-tA)}=
 {1\over det (A) ~t^2-tr (A) ~t+1}. 
\end{equation} 

\bigskip
Corollary \ref{cor1.3} follows from formula (\ref{eq3.7}). 

\section{Example}
We conclude by an example illustrating  theorem \ref{thm1.1} and corollaries
\ref{cor1.2}-\ref{cor1.3}. For simplicity,   we consider 
the Latt\`es map coming from an elliptic curve with complex multiplication. 
In this case,  an explicit formula for the functor $F: \mathscr{L}\to \mathscr{O}$
can be derived from example \ref{exm1.2}.

\begin{example}
Let  $\mathscr{E}_{CM}$ be an elliptic curve  with  complex multiplication by $\sqrt{-2}$  
given by the equation:
\begin{equation}
y^2=x^3+4x^2+2x. 
\end{equation}

\bigskip
(i)  One can  write $\mathscr{E}_{CM}\cong\mathbf{C}/L_{CM}$,
where $L_{CM}=\mathbf{Z}+\mathbf{Z}\sqrt{-2}$ is a lattice in the complex plane 
$\mathbf{C}$. 
Let $\psi: \mathscr{E}_{CM}\to \mathscr{E}_{CM}$ be an
endomorphism defined as  multiplication of the 
$L_{CM}$ by  $\sqrt{-2}$.  The norm $N(\sqrt{-2})=2$ and therefore the degree
of $\psi$ is equal to $2$. The duplication formulas (\ref{eq1.1}). with $a=4, ~b=2$ and $c=0$
 imply that the corresponding Latt\`es map has the form:
\begin{equation}\label{eq4.2}
\phi(x)={(x^2-2)^2\over x(x^2+4x+2)}. 
\end{equation}

\bigskip
(ii) Recall that $\mathscr{F}(\mathscr{E}_{CM})=\mathscr{A}_{\sqrt{2}}$ (example \ref{exm1.2}). 
The endomorphism $\psi$ transforms into an endomorphism of the pseudo-lattice 
$\Lambda=\mathbf{Z}
+\mathbf{Z}\sqrt{2}\subset \mathbf{R}$ given as multiplication by $\sqrt{2}$. 
Denoting  by $N$ and $Tr$  the norm and the trace of an algebraic number, one gets
the matrix form of the algebraic number:
\begin{equation}\label{eq4.3}
A=\left(\begin{matrix} 0 & 1\cr -N(\sqrt{2}) & Tr(\sqrt{2})\end{matrix}\right)=
\left(\begin{matrix} 0 & 1\cr 2 & 0\end{matrix}\right).
\end{equation}

\bigskip
(iii) It is easy to see, that $\sqrt{2} ~\Lambda=2\mathbf{Z}+\mathbf{Z}\sqrt{2}$ and therefore
$\Lambda_A\cong \mathbf{Z}+\mathbf{Z}{\sqrt{2}\over 2}$. 
Since the continued fraction ${\sqrt{2}\over 2}=[0; 1, \overline{2}]$ has  period $\overline{2}$, we conclude that
\begin{equation}\label{eq4.4}
T=
\left(\begin{matrix} 2 & 1\cr 1 & 0\end{matrix}\right).
\end{equation}

\bigskip
(iv) The zeta function of the Latt\`es map (\ref{eq4.2}) can be calculated by a substitution of the matrix 
(\ref{eq4.3}) in  the formula (\ref{eq1.4}): 
\begin{equation}\label{eq4.5}
\zeta_{\phi}(t)={1\over 1-2t^2}. 
\end{equation}
\end{example}

\bibliographystyle{amsplain}


\end{document}